\documentclass[12pt]{article}
\usepackage{graphicx} 
\usepackage[utf8]{inputenc}

\usepackage{amsmath, amsfonts, amssymb, latexsym, amsthm}
\usepackage{xcolor}
\usepackage{soul}
\usepackage{algorithm}
\usepackage{algpseudocode}
\usepackage{esvect}
\usepackage{graphicx}
\usepackage{setspace}
\usepackage{url}
\usepackage{hyperref}
\usepackage{wrapfig}
\usepackage{tikz}
\usetikzlibrary{calc}
\usetikzlibrary{positioning}
\usepackage{tkz-graph}
\usetikzlibrary{shapes.geometric}
\usetikzlibrary{decorations.markings}
\usetikzlibrary{decorations.pathreplacing}
\usetikzlibrary{patterns}
\usepackage{mathtools}
\theoremstyle{definition}
\newtheorem{theorem}{Theorem}[section]
\newtheorem{lemma}[theorem]{Lemma}

\usepackage{comment}

\begin{document}

\title{Eternally surrounding a robber}

\date{}

\author{
Nancy E.~Clarke\thanks{
Department of Mathematics and Statistics, 
Acadia University, Wolfville, NS, Canada.
Research support by NSERC grant \#2020-06528.
}
\and
Danny Dyer \thanks{
Department of Mathematics and Statistics,
Memorial University of Newfoundland, St.~John's, NL, Canada. Supported by NSERC.
}
\and
William Kellough \thanks{Department of Mathematics and Statistics,
Memorial University of Newfoundland, St.~John's, NL, Canada. Supported by NSERC Canada Graduate Scholarship-Master’s}
}

\maketitle

\begin{abstract}
We introduce the bodyguard problem for graphs. This is a variation of Surrounding Cops and Robber but, in this model, a smallest possible group of bodyguards must surround the president and then maintain this protection indefinitely. We investigate some elementary bounds, then solve this problem for the infinite graph families of complete graphs, wheels, trees, cycles, complete multipartite graphs, and two-dimensional grids. We also examine the problem in more general Cartesian, strong, and lexicographic products.
\end{abstract}

\begin{quote} \small {\bf Keywords:} discrete time graph processes; pursuit-evasion; Surrounding Cops and Robber

\smallskip 

\end{quote}

\normalsize

\section{Introduction}\label{sec: introduction}

Cops and Robber is a pursuit-evasion game played on the vertices of a graph where a set of cops move along edges to capture a robber. 
The classical game was first introduced by Quilliot \cite{Q78} and independently by Nowakowski and Winkler \cite{NW83}. The \emph{cop number} of a graph $G$, introduced by Aigner and Fromme \cite{AF84} and denoted $c(G)$, is the fewest cops that suffice to win Cops and Robber when playing on $G$. For a survey of known results, see \cite{BN11}. 

In Cops and Robber, the game ends when a cop moves to the vertex occupied by the robber. Suppose we were to change the win condition for the cops from capturing the robber once to capturing the robber on all but finitely many turns. After one cop has captured the robber, this new win condition is easy for the cops to satisfy since the robber has no choice but to stay within the closed neighbourhood of the cop that captured him. Therefore the fewest cops needed to win this version of Cops and Robber is the same as the cop number. 

In this paper, we apply this indefinite capturing win condition to a variant of Surrounding Cops and Robber \cite{BCCDFJP20}. In the Surrounding model, introduced in \cite{JSU23}, the cops win by surrounding the robber's position. So ``capturing" means occupying every vertex in the open neighbourhood of the robber's position. Unlike in Cops and Robber, capturing once in this Surrounding game does not guarantee that the cops can capture the robber indefinitely. Consider Figure \ref{fig: ex where indef surr matters}. If the robber is surrounded on $x$ by three cops, when the robber moves to $y$ the three cops cannot surround the robber in one move. Thus more cops are needed to win with this new condition. We call the game where the cops need to indefinitely surround the robber \emph{Bodyguards and Presidents}.

\begin{figure}[hbp]
    \centering
    \begin{tikzpicture}
    \tikzstyle{vertex} = [circle, draw=black, minimum size=20pt,inner sep=0pt]
    \tikzstyle{cop}=[circle, very thick, fill=blue!5, draw=blue, minimum size=20pt,inner sep=0pt]
    \tikzstyle{r}=[circle, very thick, fill=red!5, draw=red, minimum size=20pt,inner sep=0pt]

    \node[r, xshift=-4cm, yshift=-7cm] at (0,0) (x1) {R};
    \node[cop, xshift=-4cm, yshift=-7cm] at (-1.5,1.5) (x2) {C};
    \node[cop, xshift=-4cm, yshift=-7cm] at (-1.5,-1.5) (x3) {C};
    \node[cop, xshift=-4cm, yshift=-7cm] at (1.5,0) (x4) {C};
    \node[vertex, xshift=-4cm, yshift=-7cm] at (3,1.5) (x5) {};
    \node[vertex, xshift=-4cm, yshift=-7cm] at (3,-1.5) (x6) {};

    \node[xshift=-4cm, yshift=-7.65cm] at (0,0) {$x$};
    \node[xshift=-4cm, yshift=-7.65cm] at (1.5,0) {$y$};

    \draw (x2)--(x1)--(x4)--(x5);
    \draw (x1)--(x3);
    \draw (x4)--(x6);

    \node[yshift=-7cm] at (-0.5,0) (tail) {};
    \node[yshift=-7cm] at (2,0) (head) {};
    \draw[ultra thick, ->] (tail)--(head);

    \node[cop, xshift=4cm, yshift=-7cm] at (0,0) (x1) {CC};
    \node[vertex, xshift=4cm, yshift=-7cm] at (-1.5,1.5) (x2) {};
    \node[vertex, xshift=4cm, yshift=-7cm] at (-1.5,-1.5) (x3) {};
    \node[r, xshift=4cm, yshift=-7cm] at (1.5,0) (x4) {R};
    \node[cop, xshift=4cm, yshift=-7cm] at (3,1.5) (x5) {C};
    \node[vertex, xshift=4cm, yshift=-7cm] at (3,-1.5) (x6) {};

    \node[xshift=4cm, yshift=-7.65cm] at (0,0) {$x$};
    \node[xshift=4cm, yshift=-7.65cm] at (1.5,0) {$y$};

    \draw (x2)--(x1)--(x4)--(x5);
    \draw (x1)--(x3);
    \draw (x4)--(x6);
    \draw[ultra thick, blue, ->] (x2)--(x1);
    \draw[ultra thick, blue, ->] (x3)--(x1);
    \draw[ultra thick, violet, ->] (x1)--(x4);
    \draw[ultra thick, blue, ->] (x4)--(x5);
    
    \end{tikzpicture}
    \caption{A graph that needs many cops to cover a robber's neighbourhood as the robber moves.}
    \label{fig: ex where indef surr matters}
\end{figure}

Bodyguards and Presidents is a pursuit-evasion game played on the vertices of a graph. One player controls a set of tokens called bodyguards while the other player controls a token called the president. The game begins with the bodyguard tokens being placed on vertices followed by the president token. Both players take turns moving their tokens. During a player's turn, for each token they control, they are allowed to either move that token to an adjacent vertex or have that token stay at the vertex it currently occupies. The bodyguards win if, after some finite number of turns, they can surround the president at the end of every bodyguard turn that follows; that is, the bodyguards win if, for all but finitely many turns, there is a bodyguard on every vertex in the open neighbourhood of the president's position. The president wins otherwise. Following \cite{BN11}, we use she/her for the bodyguards and he/him for the president. In Bodyguards and Presidents, the bodyguards' goal does not resemble the act of stopping someone that is ``on the loose'' as in Cops and Robber. When playing to win, the bodyguards' behaviour suggests the evader is an important entity they aim to protect by constantly surrounding him. This is why we have chosen to use the term ``bodyguards" for the pursuers and ``president" for the evader.

For a given graph $G$ the \emph{bodyguard number} of $G$, denoted $B(G)$, is the fewest bodyguards that suffice to win on $G$. If $G$ is disconnected, then $B(G)$ is equal to the sum of the bodyguard numbers of its components. Thus, for many of our results, we only consider the case where $G$ is connected.

Bodyguards and Presidents was developed to obtain winning strategies for the cops in two variants, Surrounding Cops and Robbers \cite{ADG24, BCCDFJP20, BH19, CCHHM25, JSU23} and Cops and Cheating Robots \cite{CDK25, HN21}, when playing on the strong product $G\boxtimes H$. If we want to restrict the robber from accessing a subgraph of $G\boxtimes H$ that is isomorphic to $G$, then the cops must place themselves on that subgraph such that the closed neighbourhood of the robber's ``shadow" on the subgraph is indefinitely occupied by cops. Then we can take a winning strategy for the cops on $H$ and translate it into a winning strategy for the cops on $G\boxtimes H$. In \cite{CDK25, K24}, this technique is used to obtain tight upper bounds on the cop number of $G\boxtimes H$ and $\boxtimes^k_{i=1} P_{n_i}$ in the Surrounding Cops and Robbers and Cops and Cheating Robots models. It can also obtain exact values for the cheating robot numbers of $K_m \boxtimes P_n$, $K_m \boxtimes C_n$, and $K_m \boxtimes K_n$ for admissible $m$ and $n$. For more information, see \cite{CDK25, K24}. 

Studying this model is further justified by its potential applications. Bodyguards and Presidents, unlike Cops and Robber, acts as a model for indefinite surveillance. Consider a graph of bank accounts with each vertex representing an individual's bank account and edges representing two bank accounts that can have large transactions occur between them. If laundered money needed to be traced, then it would be useful to monitor the whole neighbourhood of an account with the laundered money and to continue doing so as the money moves from one account to another. 

In Section \ref{sec: prelim results}, we give some introductory results on the bodyguard number and in Section \ref{sec: graph families} we determine the bodyguard number for some families of graphs. In Section \ref{sec: char for n-1} we give a complete characterization for graphs with maximum bodyguard number. In Section \ref{sec: Cartesian prod} we discuss how the bodyguard number behaves with respect to the Cartesian product and give bounds on the bodyguard number for $k$-dimensional Cartesian grids. In Section \ref{sec: other prods} we discuss the bodyguard number of strong and lexicographic products. We end with open questions and directions for future work in Section \ref{sec: further directions}.

\section{Preliminary Results}\label{sec: prelim results}

We begin with an upper bound on the bodyguard number. 

\begin{lemma}\label{lem: trivial upper bound on B(G)}
    If $G$ is a graph on $n$ vertices, then $B(G) \leq n-1$.
\end{lemma}

\begin{proof}
    At the beginning of the game, place $n-1$ bodyguards on the graph such that only one vertex does not have a bodyguard. Regardless of where the president is placed, the bodyguards respond by making the president's vertex the only one not occupied by a bodyguard. From here, whenever the president moves to a vertex with a bodyguard on it, that bodyguard moves to the vertex the president came from. Thus the president is surrounded indefinitely.
\end{proof}

From Lemma \ref{lem: trivial upper bound on B(G)}, we have that the bodyguard number is well-defined. The strategy described in the proof of Lemma \ref{lem: trivial upper bound on B(G)} demonstrates how the game can be thought of as occurring in two phases: first, the bodyguards move to surround the president; and second they move to eternally surround him thereafter. 

Next, a trivial lower bound on the bodyguard number of a graph $G$ in terms of the maximum degree of the graph $\Delta(G)$.

\begin{lemma}\label{lem: trivial lower bounds on bodyguard num}
    For any graph $G$, $B(G) \geq \Delta(G)$.
\end{lemma}

\begin{proof}
    If the president stays on the vertex of maximum degree, the bodyguards cannot win without first surrounding the president on that vertex. 
\end{proof}

In Cops and Robber, the cop number of a graph $G$ can be less than the cop number of a subgraph $H$ of $G$. This can also happen with the bodyguard number.

\begin{theorem}\label{thm: bodyguard num not monotonic wrt subgraphs}
    There exists a graph $G$ with a connected subgraph $H$ such that $B(H) > B(G)$.
\end{theorem}

\begin{proof}
    Let $H$ be the graph in Figure \ref{fig: ex where indef surr matters} and let $G$ be the graph in Figure \ref{fig: bodyguard num bigger in subgraph}. From the discussion in Section \ref{sec: introduction}, $B(H) > 3$. 

    Let $b_1$, $b_2$ and $b_3$ be the three bodyguards. Consider the two positions $b_1$ at $v_1$, $b_2$ at $v_5$ and $b_3$ at $v_6$; and $b_1$ at $v_4$, $b_2$ at $v_2$ and $b_3$ at $v_3$. We call the positions $A_1$ and $A_2$ respectively. In Figure \ref{fig: bodyguard num bigger in subgraph}, $A_1$ is illustrated on the left and $A_2$ is illustrated on the right. Note that one of $A_1$ or $A_2$ surround each vertex, and the 
    bodyguards can move from $A_1$ to $A_2$ (and vice versa) in one move. Therefore, the bodyguards can surround the president every round indefinitely. So $B(G) \leq 3$. 
\end{proof}

\begin{figure}[htb]
        \centering
        \begin{tikzpicture}[scale=0.75]
        \tikzstyle{vertex} = [circle, draw=black, minimum size=20pt,inner sep=0pt]
        \tikzstyle{bodyguard}=[circle, very thick, fill=blue!5, draw=blue, minimum size=20pt,inner sep=0pt]
        
        \node[bodyguard] at (0,0) (v1) {$b_1$};
        \node[vertex] at (-2,2) (v2) {};
        \node[vertex] at (-2,-2) (v3) {};
        \node[vertex] at (2,0) (v4) {};
        \node[bodyguard] at (4,2) (v5) {$b_2$};
        \node[bodyguard] at (4,-2) (v6) {$b_3$};

        \node[yshift=-0.75cm] at (0,0) {$v_1$};
        \node[yshift=-0.75cm] at (-2,2) {$v_2$};
        \node[yshift=-0.75cm] at (-2,-2) {$v_3$};
        \node[yshift=-0.75cm] at (2,0) {$v_4$};
        \node[yshift=-0.75cm] at (4,2) {$v_5$};
        \node[yshift=-0.75cm] at (4,-2) {$v_6$};

        \draw (v5)--(v2)--(v1)--(v4)--(v5);
        \draw (v1)--(v3)--(v6);
        \draw (v4)--(v6);

        \node[] at (4.5,0) (tail) {};
        \node[] at (7,0) (head) {};
        \draw[ultra thick, <->] (tail)--(head);

        \node[vertex, xshift=7cm] at (0,0) (u1) {};
        \node[bodyguard, xshift=7cm] at (-2,2) (u2) {$b_2$};
        \node[bodyguard, xshift=7cm] at (-2,-2) (u3) {$b_3$};
        \node[bodyguard, xshift=7cm] at (2,0) (u4) {$b_1$};
        \node[vertex, xshift=7cm] at (4,2) (u5) {};
        \node[vertex, xshift=7cm] at (4,-2) (u6) {};

        \draw (u5)--(u2)--(u1)--(u4)--(u5);
        \draw (u1)--(u3)--(u6);
        \draw (u4)--(u6);
        \end{tikzpicture}
        \caption{The bodyguard winning strategy described in the proof of Theorem \ref{thm: bodyguard num not monotonic wrt subgraphs}.}
        \label{fig: bodyguard num bigger in subgraph}
\end{figure}

Recall that an induced subgraph $H$ of $G$ is a \emph{retract} if there exists a homomorphism $f: V(G) \to V(H)$ such that for every $v\in V(H)$, $f(v) = v$. In this case, $f$ is called the \emph{retraction map}. Beraducci and Intrigila \cite{BI93} proved that the cop number is monotonic with respect to retracts. We now show that the same holds for the bodyguard number.

\begin{theorem}\label{thm: Bodyguard number of retracts}
    If $H$ is a retract of a graph $G$, then $B(H) \leq B(G)$.
\end{theorem}

\begin{proof}

Let $r:V(G)\to V(H)$ be a retraction map from $G$ to $H$. We will show that $B(G)$ bodyguards have a winning strategy on $H$.

The president is playing on $H$, but the bodyguards consider a parallel game on $G$, where they have a winning strategy. On $H$, they play the images under $r$ of these winning moves on $G$. Suppose the president is surrounded in the parallel game in $G$ on the vertex $v\in V(H) \subseteq V(G)$. So $r(v) = v$. Since $r$ is a retraction map, for every vertex $u\in N_H(v) \subseteq N_G (v)$, $r(u) = u \in N_H(r(v)) = N_H (v)$ and so $N_H(v) \subseteq \cup_{u\in N_G(v)} r(u)$. Thus, if the president is surrounded on $G$, the president's image is surrounded on $H$, and since the image and position coincide in a retraction, the president is surrounded on $H$. Since the bodyguards surround the president for an infinite number of consecutive turns on $G$, they do the same on $H$. \qedhere
\end{proof}

\section{Characterization of Graphs with Maximum Bodyguard Number}\label{sec: char for n-1}

From Lemma \ref{lem: trivial upper bound on B(G)}, we know that for all graphs on $n$ vertices, the bodyguard number is at most $n-1$. In this section, we prove that this upper bound is reached if and only if the maximum degree is $n-1$. Lemma \ref{lem: trivial lower bounds on bodyguard num} gives one direction of the proof. We will prove the other direction holds for disconnected graphs, graphs with cut vertices, graphs with diameter two, and any graph that does not have any of these properties. 

\begin{lemma}\label{lem: disconnected implies B(G) < n-1}
    If $G$ is a disconnected graph on $n$ vertices, then $B(G) \leq n-2$.
\end{lemma}

\begin{proof}
    Let $G$ be disconnected with components $H_1, \dots, H_\ell$ where $\ell > 1$. Let $n_i = |V(H_i)|$ for each $1\leq i\leq \ell$. By Lemma \ref{lem: trivial upper bound on B(G)}, we have \[B(G) = \sum_{i=1}^\ell B(H_i) \leq \sum_{i=1}^\ell (n_i - 1) = n - \ell \leq n-2.\] 
\end{proof}

\begin{lemma}\label{lem: cut vertex and max deg < n-1 implies B(G) < n-1}
    If $G$ is a connected graph on $n$ vertices with a cut vertex and $\Delta(G) \leq n-2$, then $B(G) \leq n-2$. 
\end{lemma}

\begin{proof}
    Let $u$ be a cut vertex of $G$. Since $\Delta(G) \leq n-2$, pick $x\in V(G)$ not adjacent to $u$. Let $H_1, \dots, H_k$ be the components of $G-u$ and, without loss of generality, assume that $x \in V(H_1)$. 

Place a bodyguard on every vertex except for $u$ and $x$. If the president starts on $u$, the bodyguards do not move. If the president starts on $x$ then a bodyguard from $H_2$  moves to $u$. If the president starts on $v\in V(H_1) \backslash\{x\}$, then there exists a path $v = v_1, v_2, \dots, v_m = x$ that does not contain $u$ since $H_1$ is a component of $G-u$. The bodyguards on each $v_i$ move to $v_{i+1}$ for all $1\leq i\leq m-1$ and a bodyguard on $H_2$ moves onto $u$. If the president starts on any component other than $H_1$, then a bodyguard on $H_1$ moves to $u$. With this strategy the president is surrounded by the end of the first bodyguard turn.
    
For the rest of the game,the bodyguards respond to the president's moves in the following way. If the president moves onto $u$, then the bodyguards move so that $x$ and $u$ are the only vertices that are not occupied by a bodyguard. If the president moves from $u$ onto $H_1$, then there exists a path $x_1, x_2, \dots, x_r = x$ where $x_1$ is the vertex the president moved onto. The bodyguards on each $x_i$ move to $x_{i+1}$ for each $1\leq i\leq r-1$ and a bodyguard on $H_2$ moves to $u$. If the president moves from a vertex in $H_1$ to another vertex in $H_1$ then the bodyguards follow the strategy from the proof of Lemma \ref{lem: trivial upper bound on B(G)}. If the president moves from $u$ onto $H_i$ for any $i\neq 1$, then a bodyguard from $H_1$ moves onto $u$. If the president moves from a vertex in $H_i$ to another vertex in $H_i$ where $i\neq 1$ then the bodyguards do not move. Thus, the president is surrounded at the end of every subsequent bodyguard turn.
\end{proof}

\begin{lemma}\label{lem: diam(G) = 2 no uni vertex implies B(G) < n-1}
    If $G$ is connected on $n\geq 3$ vertices with $\text{diam}(G) = 2$ and $\Delta(G) \leq n-2$, then $B(G) \leq n-2$. 
\end{lemma}

\begin{proof}
    Let $x$ and $y$ be two nonadjacent vertices. Place a bodyguard on every vertex except $x$ and $y$. First, we show that the bodyguards can surround the president by the end of their first move. If the president starts on $x$ or $y$, then he will already be surrounded. Suppose the president starts on a vertex $v_p \in V(G) \backslash \{x,y\}$.

    There are three cases: the president is adjacent to exactly one, neither, or both of $x$ and $y$. If the president is adjacent to $x$ and not $y$ (without loss of generality), then the bodyguard on $v_p$ moves to $x$ to surround the president.  If the president is adjacent to neither $x$ nor $y$ then the president is already surrounded.

    Suppose instead that the president is adjacent to both $x$ and $y$. Since $\Delta(G) < n-1$ by assumption, there exists $w \in V(G) \backslash \{x,y\}$ that is not adjacent to $v_p$. Since $\text{diam}(G) = 2$, $d(w,x) \in \{1,2\}$ and $d(w,y) \in \{1,2\}$. If $w$ is adjacent to $x$, move the bodyguard from $w$ to $x$ and the bodyguard from $v_p$ to $y$ to surround the president. If $w$ is not adjacent to $x$ and is adjacent to $y$, move the bodyguard from $w$ to $y$ and the bodyguard from $v_p$ to $x$ to surround the president. Suppose instead that $d(w,x) = d(w,y) = 2$. Since $w$ is not adjacent to $v_p$ and $x$ is not adjacent to $y$, there exists a $w$-$x$ path on three vertices that contains neither $v_p$ nor $y$. Let $z$ be the third vertex in this $w$-$x$ path. The bodyguard on $z$ can move to $x$, the bodyguard on $w$ can move to $z$, and the bodyguard on $v_p$ can move to $y$ to surround the president in one bodyguard turn.

    Therefore, regardless of where the presidents starts the game, the bodyguards can surround him in one move. Let $x^\prime$ and $y^\prime$ be the two vertices that are not occupied by bodyguards at the end of the first bodyguard move. Now consider the position of the president by the end of his first move. The president will either be adjacent to exactly one, neither, or both of $x^\prime$ and $y^\prime$, and the bodyguards can repeat the same strategy as above. By induction, the $n-2$ bodyguards can indefinitely surround the president. 
\end{proof}

\begin{lemma}\label{lem: diam >= 3 case for B(G) = n-1 char.}
    If $G$ is a connected graph on $n$ vertices with no cut vertices and $\text{diam}(G) \geq 3$, then $B(G) \leq n-2$.
\end{lemma}

\begin{proof}
    Since $\text{diam}(G) \geq 3$, there exists two vertices $x,y \in V(G)$ that are distance at least three apart. Begin by placing a bodyguard on every vertex except $x$ and $y$. We show that the president can be surrounded such that no bodyguard is on the same vertex as the president by considering three cases based on where the president begins.

    \textbf{Case (i):} The president is on either $x$ or $y$. Then the president is surrounded, and the bodyguards do not move.
    
    \textbf{Case (ii):} The president is outside of the closed neighbourhoods of $x$ and $y$. Without loss of generality, assume that $x$ is at most as close to the president as $y$. Let $v_1, v_2, \dots, v_m$ be a path in $G$ such that $v_1$ is the vertex occupied by the president and $v_m = x$. Let $b_i$ be the bodyguard on $v_i$. On the bodyguards' first turn, for each $1\leq i\leq m-1$ the bodyguard $b_i$ moves from $v_i$ to $v_{i+1}$. This results in a bodyguard on every vertex in the president's closed neighbourhood except for the vertex he is occupying.

    \textbf{Case (iii):} The president is adjacent to exactly one of $x$ or $y$. Without loss of generality, assume the president is adjacent to $x$. Since $d(x,y) \geq 3$, every vertex in the president's closed neighbourhood besides $x$ is occupied by a bodyguard. On the first bodyguard turn, the bodyguard on the president's vertex moves to $x$. 

    So the president can be surrounded in one move such that no bodyguard is occupying the president's vertex. Now consider the $k$-th president move where $k\geq 1$. Every move the president makes from this position has
    three cases; either the president does not move, the president moves to a
    vertex that is adjacent to exactly one vertex that is not occupied by a bodyguard, or he moves to a vertex that is adjacent to two vertices that are not occupied by bodyguards. Let $v_i$ be the vertex the president starts his turn on and let $v_f$ be the vertex the president ends his turn on.

    If $v_i = v_f$, then the bodyguards do not move. If $v_i$ is the only vertex adjacent to $v_f$ that contains a bodyguard, then the bodyguard on $v_f$ moves to $v_i$ while the rest of the bodyguards do not move. 
    
    Finally, suppose instead that $v_f$ is adjacent to a second vertex other than $v_i$, call it $w$, that is not occupied by a bodyguard.  Since $\text{diam}(G) \geq 3$, $G$ has no universal vertices and hence there exists a vertex $q\notin N[v_f]$ that contains a bodyguard since $v_f$ is not universal. Without loss of generality, suppose that $d(q,w) \leq d(q, v_i)$. Since $G$ has no cut vertices, $G$ is $2$-connected. By Menger's Theorem \cite{M27}, there exists at least two internally disjoint $q$-$w$ paths. Let $P_1$ and $P_2$ be two internally disjoint $q$-$w$ paths. At most one of $P_1$ and $P_2$ contains $v_f$. Without loss of generality, assume that $P_2$ does not contain $v_f$. Let $q= x_1, x_2, \dots, x_m = w$ be the vertices of the path $P_2$. If $P_2$ does not contain $v_i$, then the bodyguards on $x_j$ can move to $x_{j+1}$ for each $1\leq j\leq m-1$ and the bodyguard at $v_f$ can move to $v_i$ to surround the president. If $P_2$ contains $v_i$, then $v_i = x_k$ for some $1\leq k\leq m-1$. In this case, the bodyguards on $x_j$ can move to $x_{j+1}$ for each $1\leq j\leq k-1$ and the bodyguard on $v_f$ can move to $w$ to surround the president. 

    Therefore, by continuing this strategy, the president is indefinitely surrounded. 
\end{proof}

We now obtain the characterization of graphs with maximum bodyguard number. 

\begin{theorem}\label{thm: char for B(G) = n-1}
    For any graph $G$ on $n$ vertices, $B(G) = n-1$ if and only if $\Delta(G) = n-1$. 
\end{theorem}

\begin{proof}
    If $\Delta(G) = n-1$ then by Lemmas \ref{lem: trivial upper bound on B(G)} and \ref{lem: trivial lower bounds on bodyguard num}, $B(G) = n-1$. Suppose instead that $\Delta(G) < n-1$. 

   Since $\Delta(G) < n-1$,  $\text{diam}(G) \geq 2$. If either $G$ is disconnected, $G$ has a cut vertex, or $\text{diam}(G) = 2$ then by Lemmas \ref{lem: disconnected implies B(G) < n-1}, \ref{lem: cut vertex and max deg < n-1 implies B(G) < n-1}, and \ref{lem: diam(G) = 2 no uni vertex implies B(G) < n-1} the result holds. If $G$ is connected, has no cut vertices, and $\text{diam}(G) \geq 3$, then the result holds by Lemma \ref{lem: diam >= 3 case for B(G) = n-1 char.}. 
\end{proof}

\section{Some Common Graph Families}\label{sec: graph families}

In this section we evaluate the bodyguard number for some well-known graph families. We note that the characterization in Theorem \ref{thm: char for B(G) = n-1} gives the bodyguard number for many classic graph families including stars, complete graphs, and wheels. Here we discuss the bodyguard numbers for cycles, complete $k$-partite graphs, and trees. 

\begin{theorem}\label{thm: Bodyguard numbers of cycles}
    If $n\geq 3$, then
    \begin{equation*}
        B(C_n) = \left\{
        \begin{array}{cc}
             2 & \text{if $n\leq 5$} \\
             3 & \text{if $n>5$.}
        \end{array}
        \right.
    \end{equation*}
\end{theorem}

\begin{proof}
    Let $v_0, v_1, \dots, v_{n-1}$ be the vertices of the cycle $C_n$ where $v_i$ is adjacent to $v_{i-1 \!\! \mod n}$ and $v_{i+1 \!\! \mod n}$. Once the president is surrounded by two bodyguards, regardless of whether the president moves clockwise or counterclockwise around the cycle, the bodyguards can follow his movements indefinitely. Thus, to determine the bodyguard number of a cycle it suffices to determine how many bodyguards are needed to surround the president.

    Consider the cycle $C_5$. Place a bodyguard on each of the vertices $v_0$ and $v_2$. Then, regardless of where the president starts, the two bodyguards surround him on their first move. Thus $B(C_5) \leq 2$, and from Lemma \ref{lem: trivial lower bounds on bodyguard num}, $B(C_5) = 2$. Note that for any two cycles $C_k$ and $C_{k-1}$, $C_{k-1}$ is a retract of $C_k$. Therefore, by Lemma \ref{lem: trivial lower bounds on bodyguard num} and Theorem \ref{thm: Bodyguard number of retracts}, $2\leq B(C_3) \leq B(C_4)\leq B(C_5) = 2.$ 

    Now consider the cycle $C_n$ for a fixed $n\geq 6$. First, we  show that two bodyguards are not enough to win the game. Suppose two bodyguards, $b_1$ and $b_2$, are playing against the president. Regardless of where the bodyguards begin the game, the president starts on a vertex that is distance at three from $b_1$. Then the president can always maintain a distance of at least two away $b_1$ by moving in the same direction as $b_1$. Therefore, $b_1$ can never be adjacent to the president and so the president never gets surrounded. Thus $B(C_n) > 2$.
    
    We give a strategy for three bodyguards, $b_1$, $b_2$ and $b_3$, on $C_n$. Note that $c(C_n) = 2$. For the bodyguards to win, they must move onto the two vertices adjacent to the president. The bodyguards use a cop winning strategy on $C_n$ to force one of the vertices adjacent to the president to be occupied by one of the bodyguards, say $b_1$. Afterwards, $b_2$ and $b_3$ use a cop winning strategy again to move $b_2$ onto the other vertex adjacent to the president, after which the president is surrounded. Therefore $B(C_n) =3$. 
\end{proof}

\begin{theorem}\label{thm: bodyguard num of complete k-partite graphs}
    If $n_1, \dots, n_k \in \mathbb{Z}^+$ with $n_1 \leq n_2 \leq \cdots \leq n_k$ then $$B(K_{n_1, \dots, n_k}) = \sum_{i=2}^k n_i.$$
\end{theorem}

\begin{proof}

    Let $X_1, \dots X_k$ be the partition of $V(K_{n_1, \dots, n_k})$ such that $|X_i| = n_i$ for each $1\leq i\leq k$. The bodyguards start on the vertices in $X_2, \dots, X_k$. Whichever $X_i$ the president resides in, the bodyguards move to the sets of vertices $X_j$ where $j\in [1,i) \cup (i,k]$. This strategy allows the bodyguards to surround the president indefinitely. Since $\Delta(K_{n_1, \dots, n_k}) = \sum_{i=2}^k n_i$, by Lemma \ref{lem: trivial lower bounds on bodyguard num} we have $B(K_{n_1, \dots, n_k}) = \sum_{i=2}^k n_i$. 
\end{proof}

\begin{theorem}\label{thm: bodyguard num of trees}
    If $T$ is a tree on at least three vertices with $\ell$ leaves, then $B(T) = \ell$.
\end{theorem}

\begin{proof}
    We begin by describing a winning strategy for $\ell$ bodyguards. Let $v_1,\dots, v_\ell$ be the leaves of $T$ and let $v_P$ denote the vertex that is occupied by the president. Let $B_1,\dots, B_\ell$ be the bodyguards and for each $i$, $B_i$ starts on $v_i$. For each $1\leq i\leq \ell$, let $P_i$ denote the unique geodesic from $v_i$ to $v_P$ and $u_i$ denote the unique vertex on $P_i$ that is adjacent to $v_P$. Note that the $u_i$'s  are not necessarily distinct and as the president moves, $P_i$ and $u_i$ may change.

    Since $c(T)=1$, for each $1\leq i\leq \ell$ the bodyguard $B_i$ can move to $u_i$ in finitely many moves. Thus $u_1\dots, u_\ell$ can be occupied indefinitely by the bodyguards after finitely many turns. Next we claim that $N(v_P) = \{u_1,\dots, u_\ell\}$. By definition of $u_i$, $\{u_1,\dots, u_\ell\} \subseteq N(v_P)$. Suppose there exists $x\in N(v_P)$ such that $x\neq u_i$ for any $i$. Note that $x \notin \{v_1,\dots, v_\ell\}$ since otherwise the edge $x v_P$ would be a geodesic from a leaf to $v_P$ and so $x \in \{u_1,\dots, u_\ell\}$ by definition. Let $P_x^i$ be the path from $x$ to $v_i$. If $v_P \notin V(P_x^i)$ for some $1\leq i\leq \ell$, then $P_x^i$ and $P_i$ share a vertex $y$. Thus the path contained in $P_x^i$ from $x$ to $y$, the path contained in $P_i$ from $y$ to $v_P$, and the edge $x v_P$ form a cycle; a contradiction. If for all $1\leq i\leq \ell$, $P_x^i$ contains $v_P$ then the only vertex adjacent to $x$ is $v_P$. So $x$ is a leaf which we have already shown leads to a contradiction.

    Thus, after finitely many turns, every vertex adjacent to the president can be indefinitely occupied by the $\ell$ bodyguards and so $B(T) \leq \ell$. Next we show that $B(T) > \ell - 1$.
    
    Let $c$ be any fixed vertex of $T$ that is not a leaf. Assume that there are $\ell - 1$ bodyguards in play. The president begins on $c$. If the bodyguards cannot surround the president at $c$, then $\ell - 1$ bodyguards are not  enough to win. Suppose instead that the bodyguards surround the president on $c$ after finitely many moves. Since there are $\ell - 1$ bodyguards and $\ell$ leaves in $T$, there exists a component $X$ in $T\backslash v_P$ that contains more leaves than bodyguards. On his next turn, the president moves onto $X$. The president repeats this process without traversing the same edge during two consecutive moves until the president moves to a vertex $u$ adjacent to the leaves $v_{n_1}, \dots, v_{n_k}$. Since the president always moves onto a component of $T\backslash v_P$ that contains fewer bodyguards than vertices in the set $\{v_1, \dots, v_\ell\}$, once the president moves onto $u$ there are at most $k-1$ bodyguards either on $u$ or on one of the $v_{n_i}$'s. Thus on the bodyguards' next move, they cannot occupy all of the leaves $v_{n_1}, \dots, v_{n_k}$ and so the president is not surrounded. Then the president can move back to $c$ and repeat this process indefinitely so that the bodyguards cannot win. So $B(G) > \ell - 1$ and thus we have $B(T) = \ell$. 
\end{proof}

Combining these results with Lemma \ref{lem: trivial lower bounds on bodyguard num}, we obtain a complete characterization of connected graphs with bodyguard numbers one and two. 

\begin{theorem}\label{thm: low bodyguard num characterizations}
    Let $G$ be a connected graph. Then $B(G) = 1$ if and only if $G\cong P_2$. Furthermore, $B(G) = 2$ if and only if either $G \cong P_n$ for some $n\geq 3$ or $G\cong C_m$ where $3\leq m\leq 5$.
\end{theorem}

\begin{proof}
    If $B(G) = 1$, then $\Delta(G) = 1$ by Lemma \ref{lem: trivial lower bounds on bodyguard num}. The only connected graph with maximum degree one is $P_2$ and so $G\cong P_2$. From Theorem \ref{thm: bodyguard num of complete k-partite graphs}, $B(P_2) = 1$. 
    
    If $B(G) = 2$, then $\Delta(G) \leq 2$. Since $\Delta(G) = 1$ implies $G \cong P_2$ and since $B(P_2) = 1$, any graph with a bodyguard number of two has $\Delta(G) = 2$. The only connected graphs with maximum degree two are paths and cycles. From Theorem \ref{thm: bodyguard num of trees}, all paths on at least three vertices have a bodyguard number of two and from Theorem \ref{thm: Bodyguard numbers of cycles}, the cycles with a bodyguard number of two are $C_3$, $C_4$ or $C_5$. 
\end{proof}

\section{Cartesian Product}\label{sec: Cartesian prod}

In this section, we present a variety of results involving Cartesian products. The \emph{Cartesian product} of graphs $G$ and $H$, denoted $G\square H$, has vertex set $V(G) \times V(H)$ and edges 
$\{(u,v), (x,y)\}$ where $ux \in E(G)$ and $v=y$, or $u=x$ and $vy\in E(H)$.
The subgraph induced by the vertices $\{(u,v) \mid u\in V(G)\}$ is be denoted $G[v]$. (Similarly for $H[u]$, $u \in V(G)$.) We think of $G[v]$ and $H[u]$ as being ``copies" of $G$ and $H$ respectively within the graph $G\square H$. We say that two copies of $G$, $G[v_1]$ and $G[v_2]$, are adjacent if $v_1 v_2 \in E(H)$. (Similarly for two copies of $H$.)

We first make use of the cop number to obtain the following bound on the bodyguard number of the Cartesian product of two graphs.

\begin{theorem}\label{thm: Cartesian product upper bound}
    If $G$ and $H$ are connected graphs, then $$B(G\square H) \leq B(G) + B(H) + c(G\square H) - 1.$$
\end{theorem}

\begin{proof}
    Let $(u_p, v_p)$ denote the president's vertex on $G\square H$. If the bodyguards are able to set up a bodyguard winning strategy on $G[v_p]$ with $B(G)$ bodyguards, then the president will be indefinitely  surrounded on $G[v_p]$. By ``set up a bodyguard winning strategy on $G[v_p]$" we mean that the bodyguards position themselves in finitely many moves such that the bodyguards can occupy every vertex in $N_{G[v_p]}(u_p)$ by the end of every following bodyguard turn. Similarly, if the bodyguards can set up a bodyguard winning strategy on $H[u_p]$ with $B(H)$ bodyguards, then the president will be indefinitely surrounded on $H[u_p]$. 
    
    Fix a bodyguard winning strategy on $G[v_p]$ and a bodyguard winning strategy on $H[u_p]$. If, by using a set of $B(G) + B(H)$ bodyguards, the bodyguards can set up both of these winning strategies simultaneously, then they can indefinitely surround the president since every move the president makes either puts him in a new copy of $G$ or a new copy of $H$ and he cannot move to both a new copy of $G$ and $H$ in one move. 
    
    The vertices that the $B(G) + B(H)$ bodyguards need to occupy to successfully implement these strategies change relative to the president. Informally, each of the vertices that the $B(G) + B(H)$ bodyguards need to occupy move the same way a robber does in a game of Cops and Robber. By using $c(G\square H) - 1$ bodyguards plus $1$ bodyguard from a set of $B(G) + B(H)$ bodyguards, the bodyguards can use a cop winning strategy to move a bodyguard, call it $b$, onto a desired vertex relative to the president's position after a finite number of moves. That is, the set of $c(G\square H)$ bodyguards ``captures" the desired vertex with $b$ being the bodyguard that moves onto the desired vertex. If the president moves in such a way that the desired vertex changes position, $b$ moved onto the new position of the desired vertex. When a bodyguard is set up in this way, we say that the vertex is \emph{captured} by the bodyguard. Once one vertex is captured, this leaves $c(G\square H) - 1$ bodyguards to join with another one of the $B(G) + B(H)$ bodyguards to capture another vertex needed for the winning strategy. 
    
    Therefore in finitely many turns, the $B(G) + B(H)$ bodyguards can position themselves to indefinitely surround the president on both $G[v_p]$ and $H[u_p]$ simultaneously. By the discussion at the beginning of the proof, this means that the $B(G) + B(H)$ bodyguards can indefinitely surround the president on $G\square H$. 
\end{proof}

Next, we investigate $2$-dimensional Cartesian grids.

\begin{lemma}\label{lem: bguard num 2xn Cartesian grid}
    If $n\geq 2$ then 
    \[
    B(P_2 \square P_n) = \begin{cases}
	2                  	&\text{if $n = 2$,}\\
	3					&\text{if $n\geq 3$.}
    \end{cases}
    \]
\end{lemma}

\begin{proof}
    If $n=2$, then $P_2 \square P_2$ is isomorphic to $C_4$. From Theorem \ref{thm: Bodyguard numbers of cycles}, $B(C_4) = 2$. 
    
    Suppose $n\geq 3$. Since $\Delta(P_2 \square P_n) = 3$, $B(P_2 \square P_n) \geq 3$. To show that three bodyguards  win on $P_2 \square P_n$, assume that the bodyguards play on $P_2 \square P_{n+2}$ with vertices labelled $(x,y)$ where $1\leq x\leq 2$ and $0\leq y\leq n+1$ while the president plays on the subgraph induced by the vertices $\{(u,v) \mid 1\leq u\leq 2, 1\leq v\leq n\}$. An example of this labelling system can be seen in Figure \ref{fig: P3 x Pn winning positions}. Note that two vertices $(x,y)$ and $(x^\prime, y^\prime)$ are adjacent if either $x=x^\prime$ and $|y-y^\prime| = 1$ or $|x - x^\prime| = 1$ and $y= y^\prime$. We give a strategy for the bodyguards to win on $P_2 \square P_{n+2}$ while the president is restricted to the $P_2 \square P_n$ subgraph.
    
    Three bodyguards start on the vertices $(1,0)$, $(2,0)$ and $(1,1)$. All three bodyguards begin the game by increasing their second coordinates every round. Let $(x_p, y_p)$ denote the vertex occupied by the president at the time when the bodyguards have moved to the vertices $(1, y_p - 1)$, $(1, y_p)$ and $(2, y_p - 1)$. From here, the bodyguards' strategy changes depending on how the president moves from this position. If the president moves by increasing his second coordinate at any point, the bodyguards can respond by moving towards the president without changing their strategy. Since the president can only increase $y_p$ finitely many times, the president will eventually only have three options: decrease his second coordinate, change his first coordinate, or stay at the vertex he is currently on. There are two cases for the president's position, either he is at $(1, y_p)$ or he is at $(2, y_p)$.

    \textbf{Case (i): The president is at $(1,y_p)$.} If the president moves from $(1,y_p)$ to $(1,y_p + 1)$, then the bodyguards advance their current second coordinate by one. Eventually, the president will be unable to repeat this move. If the president stays at $(1, y_p)$, the bodyguards move to $(1, y_p - 1)$, $(1, y_p + 1)$ and $(2, y_p)$ to surround the president. If the president moves to $(1, y_p - 1)$, the bodyguards move to $(1, y_p - 2)$, $(1, y_p)$ and $(2, y_p - 1)$ to surround the president. If the president moves to $(2, y_p)$, the bodyguards move to $(2, y_p - 1)$, $(1, y_p)$ and $(2, y_p)$. From here the president can either stay at $(2, y_p)$, move to $(1, y_p)$ or move to $(2, y_p - 1)$. If the president stays at $(2, y_p)$, the bodyguards move to $(2, y_p - 1)$, $(1, y_p)$ and $(2, y_p + 1)$ to surround the president. If the president moves to $(1, y_p)$, the bodyguards move to $(1, y_p - 1)$, $(2, y_p)$ and $(1, y_p + 1)$ to surround the president. If the president moves to $(2, y_p - 1)$, the bodyguards move to $(2, y_p - 2)$, $(1, y_p - 1)$ and $(2, y_p)$ to surround the president.

    \textbf{Case (ii): The president is at $(2,y_p)$.} If the president moves from $(2,y_p)$ to $(2,y_p + 1)$, then the bodyguards advance their current second coordinate by one. Eventually, the president will be unable to repeat this move. The president is at $(2, y_p)$. If the president stays on $(2, y_p)$ then the bodyguards move to $(1, y_p - 1)$, $(2, y_p)$ and $(2, y_p - 1)$. By the symmetry of the graph, the bodyguards surround the president from this position by using a strategy similar to the one from Case (i). If the president moves to $(1, y_p)$, then the bodyguards  follow the strategy from Case (i). If the president moves to $(2, y_p - 1)$, the bodyguards move to $(2, y_p - 2)$, $(1, y_p - 1)$ and $(2, y_p)$ to surround the president.

    Therefore, in the case where the bodyguards play on $P_2 \square P_{n+2}$ while the president is restricted to the subgraph $P_2 \square P_n$,  regardless of how the president moves three bodyguards are enough to surround him, and subsequently to indefinitely surround him. 

    To obtain a winning strategy for when the bodyguards and the president are on $P_2 \square P_n$, we map the bodyguards movements in the above winning strategy for $P_2 \square P_{n+2}$ via the retraction map from $P_2 \square P_{n+2}$ to $P_2 \square P_n$.
\end{proof}

\begin{lemma}\label{lem: P3 x Pn}
    If $n\geq 2$ then
    \[
    B(P_3 \square P_n) = \begin{cases}
	3                  	&\text{if $n = 2$,}\\
	4					&\text{if $n\geq 3$.}
    \end{cases}
    \]
\end{lemma}

\begin{proof}
    If $n=2$, $B(P_3 \square P_n) = 3$ by Lemma \ref{lem: bguard num 2xn Cartesian grid}. Let $n\geq 3$. By Lemma \ref{lem: trivial lower bounds on bodyguard num}, $B(P_3 \square P_n) \geq \Delta(P_3 \square P_n) \geq 4$. To show that four bodyguards can win on $P_3 \square P_n$, we use a similar argument as in the proof of Lemma \ref{lem: bguard num 2xn Cartesian grid}. Assume the bodyguards are playing on $P_3 \square P_{n+3}$ with vertices labelled $(u,v)$ where $1\leq u\leq 3$ and $0\leq v\leq n+2$ while the president plays on the subgraph induced by the vertices $\{(u,v) \mid 1\leq u\leq 3, 2\leq v\leq n+1\}$. The bodyguards start the game on the vertices $(1,1)$, $(2,0)$, $(2,2)$ and $(3,1)$. The bodyguards' strategy begins with increasing their second coordinates every round. Let $(x_p, y_p)$ denote the vertex occupied by the president at the time when the bodyguards have moved to the vertices $(1, y_p-1)$, $(2,y_p-2)$, $(2,y_p)$ and $(3,y_p-1)$. The bodyguards' strategy changes from this position. We give two positions for the bodyguards and the president and show that once in one of the two positions, the bodyguards can indefinitely surround the president. Following this, we show that in finitely many turns the bodyguards can force the game into one of these two positions.

    First, we describe two winning positions for the bodyguards. Suppose the president is on $(1,y)$ for some $1\leq y\leq n$. We claim that if there are bodyguards occupying the vertices $(1,y-1)$, $(1,y)$, $(2,y)$ and $(1, y+1)$ as labelled * in Figure \ref{fig: P3 x Pn winning positions}, then the president loses. In this position the president can either move to $(1,y-1)$, $(1,y+1)$ or $(2,y)$. If the president moves to $(1,y-1)$ the bodyguards move to $(1,y-2)$, $(1,y-1)$, $(2,y-1)$ and $(1, y)$ to surround the president and force the game into *. If the president moves to $(1, y+1)$ the bodyguards move to $(1,y)$, $(1,y+1)$, $(2,y+1)$ and $(1,y+2)$ to force the game into *. If the president moves to $(2,y)$ the bodyguards move to $(2,y-1)$, $(1,y)$, $(3,y)$ and $(2, y+1)$ which is ** in Figure \ref{fig: P3 x Pn winning positions}. By symmetry, the same holds if the president is on $(3, y)$ and the bodyguards occupy $(3,y-1)$, $(3,y)$, $(2,y)$ and $(3, y+1)$.
    
    Now we claim that if the president is on $(2,y)$ and the bodyguards occupy $(1,y)$, $(2, y-1)$, $(3,y)$, and $(2, y+1)$ as labelled ** in Figure \ref{fig: P3 x Pn winning positions}, then the president loses. If the president moves to $(2,y-1)$, then each of the bodyguards decreases their second coordinate to achieve position ** again. If the president moves to $(2, y+1)$, then each of the bodyguards increases their second coordinate to achieve position ** again. If the president changes his first coordinate, then the bodyguards, if possible, change their first coordinate or, for the single bodyguard that cannot move in the same direction, do not move. This way, the bodyguards force * to occur, and so the president is surrounded. 
    
    Therefore, regardless of how the president moves, four bodyguards can continue to surround him indefinitely when the bodyguards are in positions * or **. Therefore, it suffices to show that they can force the president into either position * or **.

    \begin{figure}[htb]
        \centering
        \begin{tikzpicture}[scale=0.75]
            \tikzstyle{vertex}=[circle, draw=black, minimum size=20pt,inner sep=0pt]
            \tikzstyle{bodyguard}=[circle, very thick, fill=blue!5, draw=blue, minimum size=20pt,inner sep=0pt]
            \tikzstyle{president}=[circle, very thick, fill=red!5, draw=red, minimum size=20pt,inner sep=0pt]\
            \tikzstyle{bandp}=[circle, very thick, fill=violet!5, draw=violet, minimum size=20pt,inner sep=0pt]

            \foreach \x in {0,2,4}{
                \foreach \y in {0,2,4} {
                \node[vertex] at (\x,\y) (\x \y0) {};
                \node[vertex, xshift=7cm] at (\x,\y) (\x \y1) {};
                }
                \draw[thick] (\x 00)--(\x 20)--(\x 40);
                \draw[thick] (\x 01)--(\x 21)--(\x 41);
            }
            \foreach \y in {0,2,4}{
                \draw[thick] (0\y 0)--(2\y 0)--(4\y 0);
                \draw[thick] (0\y 1)--(2\y 1)--(4\y 1);
            }

            \node[xshift=-1.25cm, yshift=-0.5cm] at (000) {$(1,y-1)$};
            \node[xshift=-1cm, yshift=-0.5cm] at (020) {$(1,y)$};
            \node[xshift=-1.25cm, yshift=-0.5cm] at (040) {$(1,y+1)$};
            \node[xshift=0cm, yshift=-0.75cm] at (200) {$(2,y-1)$};
            \node[xshift=1cm, yshift=-0.75cm] at (400) {$(3,y-1)$};
            
            \node[bodyguard] at (000) {B};
            \node[bodyguard] at (040) {B};
            \node[bandp] at (020) {BP};
            \node[bodyguard] at (220) {B};
            \node[bodyguard] at (241) {B};
            \node[bodyguard] at (021) {B};
            \node[bodyguard] at (421) {B};
            \node[bodyguard] at (201) {B};
            \node[president] at (221) {P};

            \node at (0,5) (b1) {};
            \node at (2,5) (b2) {};
            \node at (4,5) (b3) {};
            \node at (0,-1) (b4) {};
            \node at (2,-1) (b5) {};
            \node at (4,-1) (b6) {};
            \node[xshift=7cm] at (0,5) (c1) {};
            \node[xshift=7cm] at (2,5) (c2) {};
            \node[xshift=7cm] at (4,5) (c3) {};
            \node[xshift=7cm] at (0,-1) (c4) {};
            \node[xshift=7cm] at (2,-1) (c5) {};
            \node[xshift=7cm] at (4,-1) (c6) {};
            \draw[thick] (b1)--(040);
            \draw[thick] (b2)--(240);
            \draw[thick] (b3)--(440);
            \draw[thick] (b4)--(000);
            \draw[thick] (b5)--(200);
            \draw[thick] (b6)--(400);
            \draw[thick] (c1)--(041);
            \draw[thick] (c2)--(241);
            \draw[thick] (c3)--(441);
            \draw[thick] (c4)--(001);
            \draw[thick] (c5)--(201);
            \draw[thick] (c6)--(401);

            \foreach \w in {0,2,4}{
                \foreach \z in {5.5,5.75,6}{
                    \node at (\w,\z) {$\bullet$};
                    \node[xshift=7cm] at (\w,\z) {$\bullet$};
                }
                \foreach \z in {-1.5, -1.75, -2}{
                    \node at (\w,\z) {$\bullet$};
                    \node[xshift=7cm] at (\w,\z) {$\bullet$};
                }
            }

            \node at (2, -3) {*};
            \node[xshift=7cm] at (2, -3) {**};
        \end{tikzpicture}
        \caption{Two winning positions for four bodyguards on $P_3 \square P_n$.}
        \label{fig: P3 x Pn winning positions}
    \end{figure}

 Suppose $x_p = 1$. If the president moves to $(1, y_p - 1)$ or $(2, y_p)$ then the bodyguards can force position * or ** in Figure \ref{fig: P3 x Pn winning positions} respectively. Suppose the president stays at $(1, y_p)$ for his next turn. Then the bodyguards position themselves to have one bodyguard at $(1, y_p - 1)$ and $(2, y_p)$ and have two bodyguards at $(2, y_p-1)$. If the president moves to $(1, y_p-1)$, the bodyguards move into position *. If the president moves to $(2, y_p)$ then the bodyguards move to $(1, y_p)$, $(2, y_p -1)$, $(2, y_p)$ and $(3, y_p)$. From here, either the president stays at $(2,y_p)$ and the bodyguards force *, the president moves to $(2, y_p - 1)$ and the bodyguards force **, or the president moves to either $(1, y_p)$ or $(3,y_p)$ and the bodyguards force *. Suppose the president again stays at $(1, y_p)$. The bodyguards then move to $(1, y_p - 1)$, $(1, y_p)$, $(2, y_p - 1)$ and $(2, y_p)$. If the president either does not move or moves to $(1, y_p - 1)$, then the bodyguards move to *. If instead the president moves to $(2, y_p)$ the bodyguards move to $(1, y_p)$, $(2, y_p -1)$, $(2, y_p)$ and $(3, y_p)$, which is a position that was covered by a previous case. If at any point the president moves to $(x, y_p + 1)$ then the bodyguards respond by increasing their second coordinates. Since the president can only increase his second coordinate finitely many times, he will eventually be forced to move in some other direction, allowing the bodyguards to continue the above strategy.

    Suppose instead $x_p =2$ or $x_p = 3$. If $x_p = 3$ then a similar argument to above holds by the symmetry of the graph. If $x_p = 2$ then the bodyguards move into position **. 

    To obtain a winning strategy for when the bodyguards and the president are on $P_3 \square P_n$, we map the bodyguards movements in the above winning strategy for $P_3 \square P_{n+3}$ via the retraction map from $P_3 \square P_{n+3}$ to $P_3 \square P_n$.
\end{proof}

\begin{lemma}\label{lem: P4 x P4}
    $B(P_4 \square P_4) = 5$.
\end{lemma}

\begin{proof}
    By Theorem \ref{thm: Cartesian product upper bound}, $B(P_4 \square P_4) \leq 5$. Assume that four bodyguards can surround the president. We label the vertices $(x,y)$ where $1\leq x, y,\leq 4$ and $(x_1, y_1)$ is adjacent to $(x_2, y_2)$ if either $x_1 = x_2$ and $|y_1 - y_2| = 1$ or $|x_1 - x_2| = 1$ and $y_1 = y_2$. The president will stay on the vertices $(2,2)$, $(2,3)$, $(3,2)$, $(3,3)$. If the bodyguards win this game, they must surround the president in finitely many turns. Consider the last move the president makes from $(u_p, v_p)$ to $(u^\prime_p, v^\prime_p)$ before he is surrounded. The four bodyguards need to occupy the four vertices adjacent to the president after he makes his final move. Therefore, if any bodyguard is distance three away from $(u_p, v_p)$, then the president can avoid being surrounded by staying on $(u_p, v_p)$ for his final turn, so we assume that every bodyguard is at most distance two from the president before his final move. 

    Let $S$ be the set of vertices $\{(2,2), (2,3), (3,2), (3,3)\}$. Without loss of generality, assume the president is on $(2,2)$.  First, we consider the case where a vertex adjacent to the president contains two or more bodyguards. If the president does not move, then it is not possible for any of these bodyguards to move to a different vertex adjacent to the president in one move and complete the surrounding. Thus, at this point, every vertex strictly adjacent to the president contains at most one bodyguard. We consider multiple cases: there are between zero and four bodyguards adjacent to the president, and each of these bodyguards are either in $S$ or not.
    
First, we consider the case where there is a bodyguard that is neither adjacent to the president nor in $S$. Since the bodyguards are at most distance two away from the president, that bodyguard is either on $(1,1)$, $(1,3)$, $(2,4)$, $(3,1)$ or $(4,2)$. If there is a bodyguard on $(1,1)$, $(1,3)$ or $(2,4)$, the president can move to $(3,2)$ and now that bodyguard cannot move to a vertex adjacent to the president on their next move. If that bodyguard is on $(3,1)$ or $(4,2)$, then the president can move to $(2,3)$ and again that bodyguard cannot move to a vertex adjacent to the president in one move. 
    
    Suppose instead that every bodyguard is either adjacent to the president or in $S$. Since the president is one move away from being surrounded, there can be anywhere between zero and three bodyguards adjacent to the president. So, we have the following four cases:
    \begin{itemize}
        \item[(i):] there are exactly three bodyguards adjacent to the president that may or may not be in $S$ and the remaining bodyguard is in $S$ and not adjacent to the president, 
        \item[(ii):] there are exactly two bodyguards adjacent to the president that may or may not be in $S$ and the remaining two bodyguards are in $S$ and not adjacent to the president, 
        \item[(iii):] there is exactly one bodyguard adjacent to the president that may or may not be in $S$ and the remaining three bodyguards are in $S$ and not adjacent to the president, and 
        \item[(iv):] no bodyguards are adjacent to the president, but all four are in $S$. 
    \end{itemize}

    \textbf{Case (i):} Let $b$ be the bodyguard not adjacent to the president. The bodyguard $b$ is either on $(2,2)$ or $(3,3)$. Suppose either $(2,3)$ or $(3,2)$ is unoccupied. If the president moves to whichever of $(2,3)$ or $(3,2)$ is unoccupied, then the bodyguards cannot cover $(2,4)$ and $(4,2)$ respectively. Suppose, on the other hand, that there is a bodyguard on $(2,3)$ and $(3,2)$. Then, without loss of generality, the third bodyguard adjacent to the president is on $(1,2)$. Since there is no bodyguard on $(2,1)$, the president moves to $(3,2)$ and the bodyguards cannot cover $(3,1)$ and $(4,2)$ in one move.

    \textbf{Case (ii):} There are three subcases: either there is no bodyguard on $(2,2)$, there is one bodyguard on $(2,2)$, or there are two bodyguards on $(2,2)$. Suppose there are no bodyguards on $(2,2)$. Then the two bodyguards not adjacent to the president but still in $S$ are on $(3,3)$. Furthermore, since the president can be surrounded in one move by remaining at $(2,2)$, there must be bodyguards on $(1,2)$ and $(2,1)$. If the president moves to $(3,2)$ then no bodyguard can cover $(4,2)$.
    
    Suppose instead that one bodyguard is on $(2,2)$, forcing the other bodyguard onto $(3,3)$. Note that the two bodyguards adjacent to the president may or may not be on $S$. If the president remains on $(2,2)$, one of the two bodyguards adjacent to the president is forced to be on either $(1,2)$ or $(2,1)$. If there is a bodyguard on $(1,2)$ and $(2,1)$ then the president movea to $(3,2)$ and avoids being surrounded since the vertex $(4,2)$ is not covered. Suppose there is a bodyguard on $(1,2)$ and not $(2,1)$. Then the fourth bodyguard is either on $(2,3)$ or $(3,2)$.  If the president moves to whichever of these is not occupied, then either $(4,2)$ or $(2,4)$ respectively will not be covered. By symmetry, the same is true if instead there is a bodyguard on $(2,1)$ and not $(1,2)$. Therefore, there cannot be one bodyguard on $(2,2)$.
    
    Finally, suppose the two bodyguards not adjacent to the president are on $(2,2)$. If the president moves to $(3,2)$, then neither of the bodyguards on $(2,2)$ can get to $(3,1)$, $(3,3)$ or $(4,2)$ in one move. So it is not possible for all four bodyguards to cover those vertices in one move. Thus the president cannot be surrounded in one move.

    Case (iii) and Case (iv) follow similarly. As in Case (ii), we consider all combinations of possible bodyguard positions. We conclude that four bodyguards are not sufficient to surround the president, and the result follows.
\end{proof}

To determine the bodyguard numbers of all $2$-dimensional Cartesian grids, we make use of the following result  on the cop number of the Cartesian product of trees. \cite{MM87}

\begin{theorem}\label{thm: cop num of Cartesian prod of trees}
    \cite{MM87} If $T_1, \dots, T_k$ are trees such that each $T_i$ has more than one vertex, then $$c(\square^k_{i=1} T_i) = \left\lceil\frac{k+1}{2}\right\rceil.$$
\end{theorem}

Combining the results from this section, we obtain the following.

\begin{theorem}\label{thm: bodyguard num of 2d Cartesian grid}
    If $n\leq m$, then 
    \[
    B(P_n \square P_m) = \begin{cases}
	2                  	&\text{if $n, m = 2$,} \\
    3                   &\text{if $n = 2$ and $m \geq 3$,} \\
    4                   &\text{if $n=3$ and $m \geq 3$,} \\
	5					&\text{if $n, m\geq 4$.}
    \end{cases}
    \]
\end{theorem}

\begin{proof} The cases when $n \le 3$ are handled by Lemmas~\ref{lem: bguard num 2xn Cartesian grid} and \ref{lem: P3 x Pn}. 

    Now suppose $n\geq 4$. By Theorem \ref{thm: cop num of Cartesian prod of trees}, $c(P_n \square P_m) =  2$. By Theorem \ref{thm: Cartesian product upper bound}, we have that \[B(P_n \square P_m) \leq B(P_n) + B(P_m) + c(P_n\square P_m) - 1 = 2+2+2-1=5.\] The graph $P_{4} \square P_4$ is a retract of $P_n \square P_m$, so by Theorem~\ref{thm: Bodyguard number of retracts} and Lemma~\ref{lem: P4 x P4} we obtain \[5 = B(P_4 \square P_4) \leq B(P_n \square P_m) \leq 5\] and so equality holds. 
\end{proof}

We note that Theorem \ref{thm: bodyguard num of 2d Cartesian grid} gives an infinite family of graphs for which the bound in Theorem \ref{thm: Cartesian product upper bound} is tight. 

We end this section by giving upper bounds on the bodyguard numbers for $k$-dimensional Cartesian grids and hypercubes.

\begin{theorem}\label{thm: bodyguard num of k-dim cartesian grid}
    If $n_1, \dots, n_k \geq 2$, then
    \begin{equation*}
        B(\square^k_{i=1} P_{n_i}) \leq \left\lfloor \frac{5k}{2} \right\rfloor.
    \end{equation*}
\end{theorem}

\begin{proof}
    Suppose the bodyguards are playing on the graph $\square^k_{i=1} P_{n_i + 2}$ with vertices labelled $(x_1, \dots, x_k)$ where $0\leq x_i\leq n_i + 1$ while the president plays on the subgraph induced by the vertices $(y_1,\dots, y_k)$ where $1\leq y_i\leq n_i$. If the president is on the vertex $(v_1,\dots, v_k)$ and the bodyguards can occupy the vertices $(v_1, \dots, v_{i-1}, v_i + a_i, v_{i+1}, \dots, v_k)$ where $1\leq i\leq k$ and for each $i$, $a_i \in \{1, -1\}$, then the president can never escape being surrounded regardless of how he moves. Therefore, it suffices to show that $\left\lfloor \frac{5k}{2} \right\rfloor$ bodyguards can move from a position where the president is not surrounded to a position where the president is surrounded. 
    
    Suppose we are playing with $2k + c(\square^k_{i=1} P_{n_i}) - 1$ bodyguards. Note that $\Delta(\square^k_{i=1} P_{n_i}) = 2k$. Using a cop winning strategy, $c(\square^k_{i=1} P_{n_i})$ bodyguards can force one of the $2k$ neighbours of the president, say $b_i^+ = (v_1,\dots, v_{i-1}, v_i + 1, v_{i+1}, \dots, v_k)$, to become occupied by a bodyguard. Once $b_i^+$ is occupied by a bodyguard, the bodyguard that moved onto $b_i^+$ can always move to a new $b_i^+$ (as determined by the president's position) after every president turn. The $c(\square^k_{i=1} P_{n_i})-1$ bodyguards that did not move onto $b_i^+$ then combine with one of the remaining $2k-2$ bodyguards to force a bodyguard onto another one of the president's neighbours. They can continue this until eventually all of the president's neighbours (at most $2k$) are occupied by bodyguards.

Therefore, by the above strategy and Theorem \ref{thm: cop num of Cartesian prod of trees}, we need $$2k + \frac{k+1}{2} - 1 = \frac{5k-1}{2}$$ bodyguards if $k$ is odd and $$2k + \frac{k}{2} + 1 - 1 = \frac{5k}{2}$$ bodyguards if $k$ is even.
\end{proof}

\begin{theorem}\label{thm: bodyguard num of hypercubes}
    If $k\geq 3$, then
    \begin{equation*}
        k+1 \leq B(Q_k) \leq \left\lfloor \frac{3k}{2} \right\rfloor.
    \end{equation*}
\end{theorem}

\begin{proof}
    We can label the vertices in $Q_k$ with vectors $(a_1, \dots, a_k)$ where $a_i \in \{0,1\}$ for each $1\leq i\leq k$ such that two vertices are adjacent in $Q_k$ if they differ in exactly one of the $k$ coordinates in their labels. 

    First, we consider the lower bound. Assume that $k$ bodyguards can surround the president. At the start of the game the president places himself on a vertex $(x_1, \dots, x_k)$ such that one of the bodyguards is on $(x^\prime_1, \dots, x^\prime_k)$ where $x_i \not\equiv x^\prime_i \mod{2}$. Since $k\geq 3$, the president is beginning the game distance at least three from one of the bodyguards. Thus, the bodyguards cannot surround the president on their first move. 
    
    Suppose that the bodyguards surround the president on some move after their first move. Consider the president's last move before being surrounded. Let $(v_1, \dots, v_k)$ be the vertex occupied by the president. Since the president is not surrounded, there is a bodyguard, say $b$, that is distance two from the president. Let $(v_1, \dots, v_{i-1}, v^\prime_i, v_{i+1}, \dots, v_{j-1}, v^\prime_j, v_{j+1}, \dots, v_k)$ where $v_i \not\equiv v^\prime_i \mod{2}$ and $v_j \not\equiv v^\prime_j \mod{2}$ be the vertex occupied by $b$. To avoid being surrounded, the president moves to a vertex whose label differs from $(v_1, \dots, v_{i-1}, v^\prime_i, v_{i+1}, \dots, v_{j-1}, v^\prime_j, v_{j+1}, \dots, v_k)$ in one coordinate that is neither the $i$th coordinate nor the $j$th coordinate. Then the president is distance three from $b$, and so $b$ cannot move adjacent to the president on the next bodyguard turn. Since the degree of every vertex in $Q_k$ is $k$, all $k$ bodyguards must be adjacent to the president, a contradiction. 

    The upper bound $B(Q_k) \leq \left\lfloor\frac{3k}{2}\right\rfloor$ is obtained by using the technique from the proof of Theorem \ref{thm: bodyguard num of k-dim cartesian grid}. 
\end{proof}

\section{Other Graph Products}\label{sec: other prods}

In this section we give results for some strong and lexicographic products of graphs. 
The \emph{strong product} of $G$ and $H$, denoted $G\boxtimes H$, has vertex set $V(G) \times V(H)$ and edges
$\{(u,v), (x,y)\}$ where $ux \in E(G)$ and $v=y$,
$u=x$ and $vy\in E(H)$, or $ux \in E(G)$ and $vy\in E(H)$.
The \emph{lexicographic product} of $G$ and $H$, denoted $G\bullet H$, has the same vertex set and edges $\{(u,v), (x,y)\}$ where $u=x$ and $vy\in E(H)$, or $ux\in E(G)$.

The following theorem gives a result for the strong product of two graphs analogous to Theorem \ref{thm: Cartesian product upper bound}.

\begin{theorem}\label{thm: bodyguard num of general strong prod}
    Let $G$ and $H$ be connected graphs. If $c(G) \leq c(H)$ then 
    \begin{align*}
        B(G\boxtimes H) \leq B(G)\left( B(H) + 1\right) + B(H) + c(G) - 1.
    \end{align*}
\end{theorem}

\begin{proof}

    Let $(u_p, v_p)$ denote the president's vertex on $G\boxtimes H$. For any $x\in V(G)$, we say that $(x,v_p)$ is the \emph{shadow} of the president in $H[x]$, which is a function of the president's vertex. If the bodyguards can set up a winning strategy on $H[u_p]$ with $B(H)$ bodyguards then the president will be indefinitely surrounded on $H[u_p]$. Let $x_1, \dots, x_{B(G)}$ denote the vertices in $G$ that the bodyguards would occupy if they were implementing a winning strategy. If we also have $B(H) + 1$ bodyguards on $H[x_i]$ for each $1\leq i\leq B(G)$ such that the bodyguards are set up to indefinitely surround the shadow of the president in $H[x_i]$ and indefinitely occupy the shadow of the president within $H[x_i]$, then the bodyguards can indefinitely surround the president in $G\boxtimes H$ by making use of the winning bodyguard strategies in $G$ and $H$. In total, this takes at most $B(G)(B(H) + 1) + B(H)$ bodyguards. 

    We claim that by adding $c(G) - 1$ more bodyguards on top of the \linebreak $B(G)(B(H) + 1) + B(H)$ bodyguards, the bodyguards can set up the above strategy to begin indefinitely surrounding the president. By the structure of the strong product, the $B(G)(B(H) + 1)$ bodyguards can move into their positions on $H[x_i], \dots, H[x_{B(G)}]$ to begin indefinitely occupying the shadow of the president's closed neighbourhood in finitely many moves. By using $c(G)$ bodyguards, in a finite number of moves one bodyguard can be placed onto $H[u_p]$. By repeating this $B(H)$ times, in finitely many moves $B(H)$ bodyguards will be on $H[u_p]$. Next, the $B(H)$ bodyguards use a bodyguard winning strategy on $H[u_p]$. Afterwards, the bodyguards can begin indefinitely surrounding the president as described above.
\end{proof}

Next, we obtain the bodyguard number of any $k$-dimensional strong grid. To do this, we make use of a result by Neufeld and Nowakowski \cite{NN98}.

\begin{theorem}\label{thm: cop number for strong products}(\cite{NN98})
    If $G$ and $H$ are connected graphs, then $c(G\boxtimes H) \leq c(G) + c(H) - 1$.
\end{theorem}

\begin{theorem}\label{thm: Bodyguard number of strong grid}
    If $k\in \mathbb{Z}^+$ and $n_i \geq 3$ for $1\leq i\leq k$ then $$B(\boxtimes_{i=1}^k P_{n_i}) = 3^k - 1.$$
\end{theorem}

\begin{proof}
    Note that $B(\boxtimes_{i=1}^k P_{n_i}) \geq \Delta(\boxtimes_{i=1}^k P_{n_i}) = 3^k - 1$. We label the vertices of $B(\boxtimes_{i=1}^k P_{n_i})$ by vectors of length $k$, $(v_1, \dots, v_k)$ where $1\leq v_i \leq n_i$ for each $1\leq i\leq k$.

    To show that $3^k - 1$ bodyguards can win on $\boxtimes_{i=1}^k P_{n_i}$, we use the same technique as in the proof of Theorem \ref{thm: bodyguard num of k-dim cartesian grid}. Let $G= \boxtimes_{i=1}^k P^\prime_{n_i}$ where $P^\prime_{n_i}$ is obtained by adding a leaf to each end of the path $P_{n_i}$. The vertices of $G$ are labelled the same way as the vertices of $\boxtimes_{i=1}^k P_{n_i}$ with additional vertices that have labels containing $0$ and $n_i + 1$. The bodyguards  play on $G$ while the president plays on $\boxtimes_{i=1}^k P_{n_i}$. 

    If the president is on the vertex $(v_1, \dots, v_k)$ and the bodyguards can occupy all vertices of the form $(v_1 + a_1, \dots, v_k + a_k)$ where $1\leq i\leq k$, for each $i$, $a_i \in \{1, 0, -1\}$, and $a_i$ is nonzero for at least one $1\leq i\leq k$, then the president can never escape being surrounded regardless of how he moves. By Theorem \ref{thm: cop number for strong products}, $c(\boxtimes_{i=1}^k P_{n_i}) = 1$ and so each of the $3^k - 1$ bodyguards can use a cop winning strategy to move onto each of the president's neighbours. Once each of the president's neighbours is occupied by a bodyguard, the bodyguards can indefinitely surround the president.
\end{proof}

Using essentially the same strategy as in Theorem \ref{thm: bodyguard num of general strong prod}, we obtain the following result for lexicographic products. 

\begin{theorem}
    If $G$ and $H$ are connected graphs then
    \begin{align*}
        B(G\bullet H) \leq \min\{&B(G)|V(H)| + B(H) + c(G) - 1, \\ 
        &B(H)|V(G)| + B(G) + c(H) - 1\}.
    \end{align*}
\end{theorem}

\section{Further Directions}\label{sec: further directions}

In Sections \ref{sec: prelim results} and \ref{sec: graph families} we give a characterization of graphs with bodyguard numbers 1, 2, and $n-1$. What characterizations exist for graphs with $B(G) = k$ where $3\leq k\leq n-2$? Clarke and MacGillivray \cite{CM12} give characterizations for graphs where $k$ cops can win the original Cops and Robber game for all $k > 1$. Is it possible to obtain characterizations for graphs in which $k>2$ bodyguards can win? In this paper, we made frequent use of Lemma \ref{lem: trivial lower bounds on bodyguard num} to obtain various results on the bodyguard number. For which graphs does the equality $B(G) = \Delta(G)$ hold?

Meyniel's Conjecture, which was first mentioned in \cite{F87}, is an open conjecture in the Cops and Robber literature stating that for each $n\in \mathbb{Z}^+$, every connected graph on $n$ vertices has a cop number that is $O(\sqrt{n})$. In Bodyguards and Presidents, we know that the largest bodyguard number a graph on $n$ vertices can have is $n-1$ and, by Theorem \ref{thm: char for B(G) = n-1}, we know exactly which graphs have a bodyguard number of $n-1$. Therefore, if we define $B(n) = \max\{B(G) \mid |V(G)| = n\}$, then $B(n) = n-1$.

The more interesting open problem for Bodyguards and Presidents is determining the minimum bodyguard number among all connected graphs on a given number of vertices and a given number of edges. Define $B(n,e) = \min\{B(G) \mid |V(G)| = n, |E(G)| = e\}$. From the Handshaking Lemma and Lemma \ref{lem: trivial lower bounds on bodyguard num}, it can be shown that, $B(n,e) \geq \left\lceil \frac{2e}{n} \right\rceil$. In the case where $e=n-1$ we know that $B(P_n) = 2 = \left\lceil \frac{2(n-1)}{n} \right\rceil$. In the case where $e = \binom{n}{2}$, $G$ is a clique and so $B(G) = n-1 = \left\lceil \frac{2\binom{n}{2}}{n} \right\rceil$. So there exist examples where $B(n,e) = \left\lceil \frac{2e}{n} \right\rceil$. Is it true that for all $n-1\leq e\leq \binom{n}{2}$, $B(n,e) = \left\lceil \frac{2e}{n} \right\rceil$?

In Section \ref{sec: introduction}, we introduce the game Bodyguards and Presidents by considering a variant of Cops and Robber where the cops win if and only if they can surround the robber. Surrounding Cops and Robbers \cite{BCCDFJP20} is a variant where the the cops win if either the robber is surrounded or the robber ends his turn on a vertex occupied by a cop. Analogously, we can consider a variant of Bodyguards and Presidents where the bodyguards win if and only if they can indefinitely occupy the president's closed neighbourhood. Let $B[G]$ denote the fewest bodyguards needed to win this game on a graph $G$. For any $G$, $B(G) \leq B[G]\leq B(G) + 1$. Any graph in which $B(G) = \Delta(G)$ has $B[G] = B(G) + 1$. Are there any examples where $B(G) = B[G]$?


\begin{thebibliography}{00}


\bibitem{AF84}
M.~Aigner and M.~Fromme, A game of cops and robbers, \emph{Discrete Appl. Math.} \textbf{8} (1984) 1--11.

\bibitem{ADG24}
S.~S.~Akhtar, S.~Das, and H.~Gahlawat, Cops and Robber on butterflies, grids, and AT-free graphs, \emph{Discrete Appl. Math.} \textbf{345} (2024) 231--245.

\bibitem{BI93}
A.~Berarducci and B.~Intrigila, On the cop number of a graph, \emph{Advances in Appl. Math.} \textbf{14} (1993) 389--403.

\bibitem{BN11}
A.~Bonato and R.J.~Nowakowski, \emph{The Game of Cops and Robbers on Graphs}, Vol. 61 of Student Mathematical Library, American Mathematical Society, Providence, R.I., 2011.

\bibitem{BCCDFJP20}
A.C.~Burgess, R.A.~Cameron, N.E.~Clarke, P.~Danziger, S.~Finbow, C.W.~Jones, and D.A.~Pike, Cops that surround a robber, \emph{Discrete Appl. Math.} \textbf{285} (2020) 552--566.

\bibitem{BH19}
P.~Bradshaw and S.~A.~Hosseini, Surrounding cops and robbers on graphs of bounded genus, preprint arXiv:1909.09916, (2019).

\bibitem{CCHHM25}
N.E.~Clarke, D.~Cox, M.A.~Huggan, S.~Huntemann, and T.G.~Marbach, Time-delayed Cops and Robbers, \emph{Discrete Appl. Math.} \textbf{360} (2025) 394--405.

\bibitem{CDK25}
N.E.~Clarke, D.~Dyer, and W.~Kellough, Cops against a cheating robber, preprint arXiv:2409.11581, (2025).

\bibitem{CM12}
N.E.~Clarke and G.~MacGillivray, Characterizations of $k$-copwin graphs, \emph{Discrete Math.} \textbf{312} (2012) 1421--1425.

\bibitem{F87}
P.~Frankl, Cops and robbers in graphs with large girth and Cayley graphs, \emph{Discrete Appl. Math.} \textbf{17} (1987) 301--305.

\bibitem{HN21}
M.A.~Huggan and R.J.~Nowakowski, Cops and an insightful robber, \emph{Discrete Appl. Math.} \textbf{295} (2021) 112--119.

\bibitem{JSU23}
P.~Jungeblut, S.~Schneider, and T.~Ueckerdt, Cops and robber--when capturing is not surrounding, \emph{Lecture Notes in Comp.~Sci.} \textbf{14093} (2023) 403--416.

\bibitem{K24}
W.~Kellough, \emph{Cops, a Cheating Robot, Bodyguards and Presidents}, M.Sc. Thesis, Memorial University of Newfoundland, 2024. 

\bibitem{MM87}
M.~Maamoun and H.~Meyniel, On a game of policemen and robber, \emph{Discrete Math.} \textbf{17} (1987) 307--309.

\bibitem{M27}
K.~Menger, Zur allgemeinen Kurventheorie, \emph{Fund. Math.} \textbf{10} (1927) 96--115.

\bibitem{NN98}
S.~Neufeld and R.J.~Nowakowski, A game of cops and robbers played on products of graphs, \emph{Discrete Math.} \textbf{186} (1998) 253--268.

\bibitem{NW83}
R.J.~Nowakowski and P.~Winkler, Vertex-to-vertex pursuit in a graph, \emph{Discrete Math.} \textbf{43} (1983) 235--239.

\bibitem{Q78}
A.~Quilliot, Jeux et pointes fixes sur les graphes, \emph{Th\`ese de 3\`eme cycle}, Universit\'e de Paris VI, 1978, 131--145.

\end{thebibliography}
\end{document}